\documentclass[12pt, a4paper]{article}
\usepackage[english]{babel}
\usepackage[iso-8859-7]{inputenc}
\usepackage{latexsym}
\usepackage{amsmath, amsthm, amssymb}
\usepackage{amssymb}
\usepackage{amsmath}
\usepackage{amsfonts}
\usepackage{amsthm}
\usepackage[all]{xy}
\usepackage{indentfirst}
\input{epsf.tex}

\newtheorem{definition}{Definition}[section]
\newtheorem{theorem}{Theorem}[section]
\newtheorem{proposition}{Proposition}[section]

\newcommand{\nn}{\mathbb{N}}
\newcommand{\cc}{\mathbb{C}}
\newcommand{\qq}{\mathbb{Q}}

\newcommand{\di}{\displaystyle}

\newcommand{\st}{\subset}

\begin{document}
\title{\bf Doubly universal Taylor series on simply connected domains }

\author{N. Chatzigiannakidou and V. Vlachou}         
\date{}
\maketitle 
\begin{abstract} In this article we deal with  the existence of doubly universal Taylor series defined on simply connected domains with respect
to any center  and 
we generalize the results of G. Costakis and N. Tsirivas for the unit disk.
\end{abstract}
\noindent
\section{Introduction }
Let $\Omega\subset \cc$ be an open set. We denote by $H(\Omega)$  the space
of functions, holomorphic in $\Omega$, endowed with the topology of uniform convergence on compacta. 
Moreover, for a compact set $K\subset\cc$, we denote
$$\mathcal{A}(K)=\{g\in H(K^o): \ g \text{ is
continuous on } K\}$$ 
and
 $$\mathcal{M}=\{K\st
\cc: K \ \text{compact set and } \ K^c \text{ connected
set}\}.$$
Let $\zeta_{0}\in\Omega$.
It is well known that a function
$f\in H(\Omega)$,  has a Taylor 
expansion around $\zeta_0$,  which is valid inside the disk of convergence.

If we denote by $=\sum_{n=0}^{N}
\frac{f^{(n)}(\zeta_0)}{n!}(z-\zeta_0)^n,$ \ \ $N=1,2,\ldots$ the partial
sums of the Taylor expansion of $f$ around $\zeta_0$, then this sequence of polynomials
may have very strong approximation properties outside $\Omega$. To be more specific, let us give
the definition of universal Taylor series, as given by V.Nestoridis:\\
A function $f\in{H(\Omega)}$ is said to belong to the collection $U(\Omega,\zeta_{0})$ of functions with universal Taylor series expansions around $\zeta_{0}$, if the partial sums $\di \{S_{N}(f,\zeta_{0})(z), N=1,2,\ldots\}$ are dense in $A(K)$, for every $K\in \mathcal{M}$ disjoint from $\Omega$ (the topology of $A(K)$ is induced by the norm
 $\di ||g||_K=\max_{z\in K}|g(z)|$).
\\V. Nestoridis \cite{N2}, \cite{N3} has shown that $U(\Omega,\zeta_{0})\neq{\emptyset}$, for any simply connected domain $\Omega$ and any point $\zeta_{0}\in\Omega$. In particular, he showed  that the collection $U(\Omega,\zeta_{0})$ is a dense, $G_{\delta}$ subset of the space $H(\Omega)$.

Recently, G. Costakis and N. Tsirivas in \cite{c-t} worked for $\Omega=\mathbb{D}$ (the open unit disk) and  introduced a stronger notion of universality with respect to Taylor series. Let us give the definition of their class.

\begin{definition} Let $(\lambda_{n})$ be a strictly increasing sequence of positive integers. A function $f\in{H(\mathbb{D})}$ belongs to the class $U_{\text{CT}}(\mathbb{D},(\lambda_{n}))$ if 
the set $ \{(S_n(f,0)(z),S_{\lambda_n}(f,0)(z)), \ n=1,2,\ldots\}$ is dense in   ${A(K)\times{A(K)}}$, for every $K\in \mathcal{M}$ disjoint from $\mathbb{D}$. 
Such a function $f$ will be called doubly universal Taylor series with respect to the sequences $(n)$, $(\lambda_{n})$.
\end{definition}

G. Costakis and N.Tsirivas proved that the class $U_{\text{CT}}(\mathbb{D},(\lambda_{n}))$ is non-empty, if and only if,  $\di \limsup_{n}\frac{\lambda_{n}}{n}=+\infty$.\\
We study a  class of functions, with even stronger approximation properties. Namely, 

\begin{definition} Let $(\lambda_{n})$ be a strictly increasing sequence of positive integers, $\Omega$ a simply connected domain and $\zeta_{0}\in{\Omega}$. A function $f\in{H(\Omega)}$ belongs to the class $U_{CV}(\Omega,(\lambda_{n}),\zeta_{0})$ if $ \{( S_n(f,\zeta_0)(z),S_{\lambda_n}(f,\zeta_0)(z)), \ n=1,2,\ldots\}$  is dense in   ${A(K_1)\times{A(K_2)}}$ for every pair of  sets $K_{1},\,K_{2}\in\mathcal{M}$, disjoint from $\Omega$.
\end{definition}
We improve the method used in \cite{c-t} in order to obtain the result for different compact sets $K_1,K_2$ and we also overcome some technical difficulties
to generalize it to any simply connected domain $\Omega$. Our  main result is 
again that $U_{\text{CV}}(\Omega,(\lambda_{n}))$ is non-empty, if and only if,  $\di \limsup_{n}\frac{\lambda_{n}}{n}=+\infty$.\\
For analogous results in other type of universalities we refer to \cite{bernal}, \cite{bes}. The article of G. Costakis and N. Tsirivas has  complete bibliography on the subject, which  features  the motivation for investigating such questions.

\section{The positive result}
Let $K,\,L$ be compact sets and $f:K\to \cc$ be a continuous function defined on $K$.\\
Let $p$ be a polynomial. We denote by $deg p$ the degree of the polynomial i.e. the  highest degree of its terms and by $deg^- p$ the  lowest degree of its (non-zero) terms.
Then for any choice of integers  $n>m> 1$ we use the notation:
$$d_{n,m}(f,K,L)=\inf\{\max\{||f-p||_K,\,||p||_{L}\}:  p \text{ polynomial such that }$$
$$deg^- p\geq m \text{ and } deg p\leq n \}. $$
The following proposition is crucial for our result. It is a slight modification of the corresponding proposition in \cite{c-t}.
\begin{proposition}\label{B-Wtype} Let $K,L\in \mathcal{M}$ be two disjoint compact sets such that $0\in L^o$.  Let, in addition, $\{\tau_n\}_{n\in\nn}$ and 
 $\{\sigma_n\}_{n\in\nn}$ be two sequences of positive integers such that $\displaystyle \frac{\tau_n}{\sigma_n}\to +\infty$ and $U\subset\cc$ open, $K\st U$.
Then there exits $\theta\in (0,1)$  such that for every function $f$  holomorphic
in  $U$ 
$$\limsup_n d_{\tau_n,\sigma_n}(f,K,L)^{\frac{1}{\tau_n}}<\theta.$$
\end{proposition}
\begin{proof} It is similar to the proof of theorem 2.1 in \cite{c-t}.
\end{proof}
\begin{theorem} Let $\Omega$ be a simply connected domain and $\zeta_{0}\in\Omega$. Consider $(\lambda_{n})_{n\in\nn}$ a strictly increasing sequence of positive integers such that $\displaystyle\limsup_{n\in\nn}\frac{\lambda_{n}}{n}=+\infty$. Then the set $U_{CV}(\Omega,(\lambda_{n}),\zeta_{0})$ is $G_{\delta}$ and dense in $H(\Omega)$.\end{theorem}
 \begin{proof} Let $(f_{j})_{j\in\nn}$ be an enumeration of all polynomials with coefficients in $\qq+i\qq$ and let $(K_{m})_{m\in\nn}$ be a sequence of sets in $\mathcal{M}$, disjoint from $\Omega$ such that the following holds: every non-empty compact set $K\subset{\cc\setminus{\Omega}}$, having connected complement, is contained in some $K_{m}$ (for the existence of such a sequence we refer to\ \cite{N3}).
\\For any positive integers $m_1,m_2,j_1,j_2,\,s,\,n,$ we denote by $E(m_{1},m_{2},j_{1},j_{2},s,n)$ the set
\\$E(m_{1},m_{2},j_{1},j_{2},s,n)=$
$$\{f\in{H(\Omega)}:||S_{n}(f,\zeta_{0})-f_{j_{1}}||_{K_{m_{1}}}<\frac{1}{s}\text{ and }||S_{\lambda_{n}}(f,\zeta_{0})-f_{j_{2}}||_{K_{m_{2}}}<\frac{1}{s}\}.$$
Using Mergelyan's theorem it is easy to prove that  $$\displaystyle U_{CV}(\Omega,(\lambda_{n}),\zeta_{0})=\bigcap_{m_{1},m_{2},j_{1},j_{2},s\in\nn}\bigcup_{n\in\nn}E(m_{1},m_{2},j_{1},j_{2},s,n).$$
(see for example similar proof in \cite{N3}).\\
Moreover,  for every  $m_{1},m_{2},j_{1},j_{2},s,n\in\nn$ the set $E(m_{1},m_{2},j_{1},j_{2},s,n)$ is open (see for example \cite{N3}).\\
Therefore, in view of Baire's Category theorem, it suffices to prove that for every choice of positive integers $m_1,m_2,j_1,j_2,\,s$ the set $\di\bigcup_{n\in\nn}E(m_{1},m_{2},j_{1},j_{2},s,n)$ is dense in $H(\Omega)$.\\
Fix $m_{1},m_{2},j_{1},j_{2},s\in\nn$. Let $\varepsilon>0$, $L\subset{\Omega}$ be a compact set and $g\in{H(\Omega)}$. Without loss of generality,we may assume that $\zeta_{0}\in{L^{o}}$ and that the compact set $L$ has connected complement (note that $\Omega$ is simpply connected).\\
We will prove the existence of a funtion  $f\in{H(\Omega)}$ and an $n\in\nn$, such that:
\begin{itemize}
\item $\di\sup_{z\in{L}}|f(z)-g(z)|<\varepsilon\\$
\item $\di\sup_{z\in{K_{m_{1}}}}|S_{n}(f,\zeta_{0})(z)-f_{j_{1}}(z)|<\frac{1}{s}\\$
\item $\di\sup_{z\in{K_{m_{2}}}}|S_{\lambda_{n}}(f,\zeta_{0})(z)-f_{j_{2}}(z)|<\frac{1}{s}.$
\end{itemize}
First, we apply  Runge's theorem to find  a polynomial $p$ such that \\
 $\displaystyle\sup_{z\in{L}}|p(z)-g(z)|<\frac{\varepsilon}{2}$ and $\displaystyle\sup_{z\in{K_{m_{1}}}}|p(z)-f_{j_{1}}(z)|<\frac{1}{2s}$.\\
 Since $\displaystyle\limsup_{n\in\nn}\frac{\lambda_{n}}{n}=+\infty$, there exists a strictly increasing sequence of positive integers $(\mu_{n})_{n\in\nn}$, such that $\displaystyle\frac{\lambda_{\mu_{n}}}{\mu_{n}}\to{+\infty}$. Thus,  $\displaystyle\frac{\lambda_{\mu_{n}}}{\mu_{n}+1}\to{+\infty}$, as $n\to{+\infty}$.\\ \\
It is easy to see that the sets $K_{m_{2}}-\zeta_{0}=\{z-\zeta_{0},\,\forall{z\in{K_{m_{2}}}}\}$ and $L-\zeta_{0}=\{z-\zeta_{0},\,\forall{z\in{L}}\}$ are compact, disjoint and they have connected complements. Applying proposition \ref{B-Wtype}, we have that:
$$\limsup_{n}[d_{\lambda_{\mu_n},\mu_n+1}(f_{j_{2}}(z+\zeta_{0})-p(z+\zeta_{0}), K_{m_{2}}-\zeta_{0},L-\zeta_{0} )]^{\frac{1}{\lambda_{\mu_n}}}<\theta,$$
for some $\theta\in (0,1)$.\\
Therefore, there exists $N\in\nn$ such that:
$$[d_{\lambda_{\mu_n},\mu_n+1}(f_{j_{2}}(z+\zeta_{0})-p(z+\zeta_{0}), K_{m_{2}}-\zeta_{0},L-\zeta_{0} )]^{\frac{1}{\lambda_{\mu_n}}}<\theta, \ n\geq N.$$
Thus, we may choose a sequence of polynomials  $P_n$, $n\geq N$, with $deg^{-}P_n \geq \mu_{n}+1$ and $deg P_n\leq \lambda_{\mu_{n}}$
such that
$$\sup_{z\in{K_{m_{2}}-\zeta_{0}}}|f_{j_{2}}(z+\zeta_{0})-p(z+\zeta_{0})-P_{n}(z)|\leq{\theta^{\lambda_{\mu_{n}}}}\text{ and }\sup_{z\in{L-\zeta_{0}}}|P_{n}(z)|\leq{\theta^{\lambda_{\mu_{n}}}}.$$
Let us fix $n_0\in\nn$, $n_0\geq N$ such that $\mu_{n_0}\geq{deg\,p(z)}$ and $\displaystyle\theta^{\lambda_{\mu_{n_0}}}<\min\{\frac{\varepsilon}{2},\frac{1}{2s}\}$.
\\We set $f(z)=P_{n_0}(z-\zeta_{0})+p(z)$.
\\Then

$S_{\mu_{n_0}}(f,\zeta_{0})(z)=S_{\mu_{n_0}}(p,\zeta_{0})(z)=p(z)$.
\\Also,

$S_{\lambda_{\mu_{n_0}}}(f,\zeta_{0})(z)=f(z).$
\\Thus,

$$\sup_{z\in{L}}|f(z)-g(z)|=\sup_{z\in{L}}|P_{n_0}(z-\zeta_{0})+p(z)-g(z)|\leq$$
$$\leq{\sup_{z\in{L}}|P_{n_0}(z-\zeta_{0})|+\sup_{z\in{L}}|p(z)-g(z)|}=$$
$$\sup_{z\in{L-\zeta_{0}}}|P_{n_0}(z)|+\sup_{z\in{L}}|p(z)-g(z)|\leq{\theta^{\lambda_{\mu_{n_0}}}+\frac{\varepsilon}{2}}<\varepsilon,$$

$$\sup_{z\in{K_{m_{1}}}}|S_{\mu_{n_0}}(f,\zeta_{0})(z)-f_{j_{1}}(z)|=\sup_{z\in{K_{m_{1}}}}|p(z)-f_{j_{1}}(z)|<\frac{1}{2s}<\varepsilon,$$

\begin{center}$\displaystyle\sup_{z\in{K_{m_{2}}}}|S_{\lambda_{\mu_{n_0}}}(f,\zeta_{0})(z)-f_{j_{2}}(z)|=\sup_{z\in{K_{m_{2}}}}|P_{n_0}(z-\zeta_{0})+p(z)-f_{j_{2}}(z)|=$

$\displaystyle=\sup_{z\in{K_{m_{2}}-\zeta_{0}}}|P_{n_0}(z)+p(z+\zeta_{0})-f_{j_{2}}(z+\zeta_{0})|\leq{\theta^{\lambda_{\mu_{n_0}}}}<\varepsilon.$\end{center}

 The function $f$ satisfies all the requirements and the result follows.

\end{proof}

\section{Negative Result}
In this section we will use the notation $U_{CT}(\Omega, (\lambda_n))$ for the corresponding class of G. Costakis and N.Tsirivas in any simply connected domain 
$\Omega$. Obviously  $U_{CV}(\Omega, (\lambda_n))\st U_{CT}(\Omega, (\lambda_n))$.
So, in order to prove that the class $U_{CV}(\Omega, (\lambda_n))$ is otherwise empty  we will prove the stronger result  that the class $U_{CT}(\Omega, (\lambda_n))$ is empty. For this we need  the following result  of $J. M\ddot{u}ller$ and A. Yavrian in \cite{M-Y}.

\begin{theorem}\label{5.3}($M\ddot{u}ller-Yavrian$) Let $\Gamma$ be a compact and connected subset of $\cc$, but not a singleton. Let $E\subset\cc$ be a closed set such that $Ε$ is non-thin at $\infty$.
Also suppose $(P_{n})$ to be a sequence of polynomials with $deg\,P_{n}\leq{d_{n}}$, for some increasing sequence of positive integers $(d_{n})$ and having the following properties:

\begin{itemize}\item[$(\alpha)$] there exists a function $f:\Gamma\to\cc$ with $$\limsup_{n\to{+\infty}}||f-P_{n}||_{\Gamma}^{1/d_{n}}<1,$$

\item[$(\beta)$] for all $z\in{E}$ $$\limsup_{n\to{+\infty}}|P_{n}(z)|^{1/d_{n}}\leq{1}.$$\end{itemize}

Then the following statement is true:

\begin{itemize}\item[$(i)$] if the sequence $(d_{n+1}/d_{n})$ is bounded, then $f$ extends to an entire function and for every compact set $K\subset{\cc}$ we have
$$\limsup_{n\to{+\infty}}||f-P_{n}||_{Κ}^{1/d_{n}}<1,$$

\item[$(ii)$] if, for arbitrary $(d_{n})$, the function $f$ is analytic on $\Gamma$, then $f$ extends to a holomorphic function having a simply connected domain of existence $G_{f}\subset\cc$ ($G_{f}$ denotes the unique largest domain on which $f$ extends as a holomorphic function; observe that $G_{f}$ exists in this case) and for every compact set $K\subset{G_{f}}$ we have 

$$\limsup_{n\to{+\infty}}||f-P_{n}||_{Κ}^{1/d_{n}}<1.$$
\end{itemize}
\end{theorem}

\begin{theorem} Let $\Omega\subset\cc$ be a simply connected domain and $\zeta_{0}\in\Omega$. If $(\lambda_{n})_{n\in\nn}$ is a strictly increasing sequence of positive integers such that $\displaystyle\limsup_{n\in\nn}\frac{\lambda_{n}}{n}<+\infty$, then $U_{CT}(\Omega,(\lambda_{n}),\zeta_{0})=\emptyset$.
\end{theorem}
\begin{proof} Assume first that $\Omega$ is not bounded.  Arguing by contradiction, suppose that there exists $f\in{U_{CT}(\Omega,(\lambda_{n}),\zeta_{0})}$.  Since the sequence $\displaystyle(\frac{\lambda_{n}}{n})_{n\in\nn}$ is bounded, there exists $C>0$ such that $\displaystyle\frac{\lambda_{n}}{n}<C$, for every $n\in\nn$.\\
We consider the sets $$E_{n}=\Omega^{c}\cap{(D(\zeta_{0},2^{C}))^{c}}\cap{\overline{D(\zeta_{0},2^{C}+n)}}, \ n=1,2,\ldots$$
(we denote by $D(z,r)$ the open disk of center $z$ and radious $r$).\\
Without loss of generality, we may assume that $E_n\ne \emptyset, \ n\geq 1$ , since this is eventually the case anyway (or by choosing suitable $C>0$).\\ 
For every $n\in\nn$, it is easy to see that the set $E_{n}\st\Omega^c$ is closed and bounded, thus is compact. Moreover, each $E_n$ has connected complement. To see this note that 
$E_{n}^{c}=\Omega\cup{D(\zeta_{0},2^{C})}\cup{\overline{D(\zeta_{0},2^{C}+n)}^{c}}$, $\zeta_0\in \Omega\cap D(\zeta_{0},2^{C})$ and $\Omega$ is unbounded.\\
Let $\displaystyle E=\bigcup_{n\in\nn}E_{n}=\Omega^{c}\cap{D(\zeta_{0},2^{C})^{c}}.$ Since $\Omega^c$ is connected, it is non-thin at $\infty$ (see page 79 theorem 3.8.3 in \cite{R} ). Because thinness is a local property (see definition page 79 in \cite{R}), $E$ is also non-thin at $\infty$. 

Now we use the fact that  $f$ belongs to the class $U_{CT}(\Omega,(\lambda_{n}),\zeta_{0})$, to fix an  $n_{1}\in\nn$, such that

$$\sup_{z\in{E_{1}\cup\{1+\zeta_{0}\}}}\left|S_{n_{1}}(f,\zeta_{0})(z)\right|<\frac{1}{2^{C+1}}\text{ and }\sup_{z\in{E_{1}\cup\{1+\zeta_{0}\}}}\left|S_{\lambda_{n_{1}}}(f,\zeta_{0})(z)-1\right|<\frac{1}{2^{C+1}}.$$
Thus,
$$\sup_{z\in{E_{1}\cup\{1+\zeta_{0}\}}}\left|S_{\lambda_{n_{1}}}(f,\zeta_{0})(z)-S_{n_{1}}(f,\zeta_{0})(z)-1\right|\leq$$

$$\leq\sup_{z\in{E_{1}\cup\{1+\zeta_{0}\}}}\left|S_{n_{1}}(f,\zeta_{0})(z)\right|+\sup_{z\in{E_{1}\cup\{1+\zeta_{0}\}}}\left|S_{\lambda_{n_{1}}}(f,\zeta_{0})(z)-1\right|<\frac{1}{2^{C}}.$$
For every $z\in\cc$, we set
 $$S_{\lambda_{n_{1}}}(f,\zeta_{0})(z)-S_{n_{1}}(f,\zeta_{0})(z)=(z-\zeta_{0})^{n_{1}+1}P_{1}(z),$$ 
where $P_1$ is a polynomial with \\ $\displaystyle deg\,P_{1}\leq{\lambda_{n_{1}}-(n_{1}+1)}<\lambda_{n_{1}}-n_{1}=n_{1}\left(\frac{\lambda_{n_{1}}}{n_{1}}-1\right)\leq{n_{1}(C-1)}.$\\
Consequently,
$$\sup_{z\in{E_{1}\cup\{1+\zeta_{0}\}}}|(z-\zeta_{0})^{n_{1}+1}P_{1}(z)-1|<\frac{1}{2^{C}}.$$
Also, since $z\in E_1$ implies that $|z-\zeta_0|>2^c$ we have:
$$\sup_{z\in{E_{1}}}|P_{1}(z)|\leq\sup_{z\in{E_{1}}}\frac{1}{|(z-\zeta_{0})^{n_{1}+1}|}\sup_{z\in{E_{1}}}|(z-\zeta_{0})^{n_{1}+1}P_{1}(z)|\leq$$
$$\leq\frac{1}{2^{C(n_{1}+1)}}\left(\sup_{z\in{E_{1}}}|(z-\zeta_{0})^{n_{1}+1}P_{1}(z)-1|+1\right)\leq\frac{1}{2^{C(n_{1}+1)}}\left(\frac{1}{2^{C}}+1\right)=$$
$$=\frac{1}{2^{C\,n_{1}}}\left(\frac{1}{4^{C}}+\frac{1}{2^{C}}\right)<\frac{1}{2^{C\,n_{1}}}.$$
Repeating the above argument, we fix an  $n_{2}\in\nn$, with $n_{2}>n_{1}$, such that
$$\sup_{z\in{E_{2}\cup\{1+\zeta_{0}\}}}\left|S_{n_{2}}(f,\zeta_{0})(z)\right|<\frac{1}{2^{C+2}}\text{ and }\sup_{z\in{E_{2}\cup\{1+\zeta_{0}\}}}\left|S_{\lambda_{n_{2}}}(f,\zeta_{0})(z)-1\right|<\frac{1}{2^{C+2}}.$$
As before, we set 
$$\frac{S_{\lambda_{n_{2}}}(f,\zeta_{0})(z)-S_{n_{2}}(f,\zeta_{0})(z)}{(z-\zeta_{0})^{n_{2}+1}}=P_{2}(z).$$
Then $P_2$ is a polynomial of degree less than $ n_{2}(C-1).$\\
Moreover as before,
$$\sup_{z\in{E_{2}\cup\{1+\zeta_{0}\}}}|(z-\zeta_{0})^{n_{2}+1}P_{2}(z)-1|<\frac{1}{2^{C+1}}.$$
And,
$$\sup_{z\in{E_{2}}}|P_{2}(z)|<\frac{1}{2^{C\,n_{2}}}.$$
Proceeding inductively, we conclude that for every $k\in\nn$, there exists $n_{k}\in\nn$, with $n_{k}>n_{k-1}$ and polynomial $P_k$, $deg P_k\leq C n_k$ such that
\begin{equation}\label{0.1+zeta0}\displaystyle\sup_{z\in{E_{k}\cup\{1+\zeta_{0}\}}}|(z-\zeta_{0})^{n_{k}+1}P_{k}(z)-1|<\frac{1}{2^{C+k-1}}\end{equation} and \begin{center}$\displaystyle\sup_{z\in{E_{k}}}|P_{k}(z)|\leq{\frac{1}{2^{C\,n_{k}}}}$.\end{center}
Let $z\in E$. Then 
 \begin{equation}\label{1.MY}\displaystyle\limsup_{k\to\infty}|P_{k}(z)|^{1/C\,n_{k}}\leq{\lim_{k\to\infty}\frac{1}{2}}=\frac{1}{2}.\end{equation}
(Note that, $(E_{n})_{n\in\nn}$ is an increasing sequence of compact sets.)\\
Also, since $\displaystyle\sup_{z\in{E_{1}}}|P_{k}(z)|\leq{\sup_{z\in{E_{k}}}|P_{k}(z)|}\leq{\frac{1}{2^{C\,n_{k}}}}$, for every $k\in\nn$  \begin{equation}\label{2.MY}\displaystyle\limsup_{k\to\infty}||P_{k}(z)||_{E_{1}}^{1/C\,n_{k}}\leq{\lim_{k\to\infty}\left(\frac{1}{2^{C\,n_{k}}}\right)^{1/C\,n_{k}}}=\frac{1}{2}<1.\end{equation}
Therefore, we may apply the theorem $M\ddot{u}ller-Yavrian$, to a compact and connected subset  $\Gamma$ of $E_{1}$ (containing more than one points), the closed set $E$, which is non-thin at $\infty$, the sequence of polynomials $P_k$, $d_k=C\,n_{k}$, $k\in\nn$,  and the function $f\equiv{0}$. From inequalities $(\ref{1.MY})$, $(\ref{2.MY})$, we observe that the conditions of the theorem are fulfilled. Since the function $f\equiv{0}$ is entire, it follows that $P_{k}\to{0}$, uniformly on all compact subsets of $\cc$. In particular, $P_{k}(1+\zeta_{0})\to{0}$, as $k\to\infty$.
On the other hand, $(\ref{0.1+zeta0})$ implies that $P_{k}(1+\zeta_{0})\to{1}$, as $k\to\infty$, which is a contradiction. As a result, in case that $\Omega$ is not bounded, the class $U_{CT}(\Omega,(\lambda_{n}),\zeta_{0})$ is empty.

\vspace{5mm}
Now if the set $\Omega$ is bounded, then  $\Omega\cup{D(\zeta_{0},2^{C})}\subset{D(0,N)}$, for some $N\in\nn$. One can apply the previous procedure for the
compact sets $E_{n}=[N+1,N+n]$, for $n>1$ and the result follows.

\end{proof}
\textbf{Acknowledgement: } We would like to thank G. Costakis and N. Tsirivas, for sharing with us the unpublished version of their work, thus informing us about these nice problems.

N. Chatzigiannakidou and V.Vlachou, \\
Department of Mathematics,\\
University of Patras,\\
26500 Patras,GREECE\\
E-mail addresses:\\
 ni.chatzig@gmail.com (N.Chatzigiannakidou ),\\ 
vvlachou@math.upatras.gr (V.Vlachou)

\end{document}